\documentclass[12pt,twoside, draft]{article}
\usepackage{ifthen}
\usepackage{amsmath}
\usepackage{amsfonts}
\usepackage{latexsym}
\usepackage{amsthm}
\usepackage{amssymb}
\usepackage{color}

\newtheorem{Theorem}{\quad Theorem}

\newtheorem{Remark}{\quad Remark}
\newtheorem{Proposition}{\quad Proposition}

\newtheorem{Corollary}[Theorem]{\quad Corollary}

\date{}

\begin{document}


\centerline{}

\centerline {\Large{\bf On the abscissas of convergence of Dirichlet series}}
\centerline {\Large{\bf with random pairwise independent exponents}}

\centerline{}

\centerline{\bf {A. O. Kuryliak}}

\centerline{}

\centerline{Department of Mechanics and Mathematics,}

\centerline{Ivan Franko National University of L'viv, Ukraine}

\centerline{kurylyak88@gmail.com}

\centerline{}

\centerline{\bf {O. B. Skaskiv}}

\centerline{}

\centerline{Department of Mechanics and Mathematics,}

\centerline{Ivan Franko National University of L'viv, Ukraine}

\centerline{olskask@gmail.com}

\centerline{}

\centerline{}

\centerline{\bf {N. Yu. Stasiv}}

\centerline{}

\centerline{Department of Mechanics and Mathematics,}

\centerline{Ivan Franko National University of L'viv, Ukraine}

\centerline{n-stas@ukr.net}

\begin{abstract}
For the Dirichlet series of the form $\displaystyle
F(z,\omega)=\sum\nolimits_{k=0}^{+\infty} f_k(\omega)e^{z\lambda_k(\omega)} $ $ (z\in\mathbb{C},$\ $\omega\in\Omega)$ with pairwise independent real exponents $(\lambda_k(\omega))$ on probability space $(\Omega,\mathcal{A},P)$ an estimates of abscissas convergence and absolutely convergence are established.
\end{abstract}



\noindent {\bf 1. Introduction.}    Let $(\Omega,\mathcal{A},P)$ {be a}  probability space,
$\mathbf{\Lambda}=\big(\lambda_k(\omega)\big)_{k=0}^{+\infty}$ and $\mathbf{f}=\big(f_k(\omega)\big)_{k=0}^{+\infty}$ sequences of positive and complex-valued random  variables on it, respectively. 
Let
$\mathcal{D}$ be the class of formal random Dirichlet series of the form

\smallskip\centerline{$\displaystyle
f(z)=f(z,\omega)=\sum\limits_{k=0}^{+\infty} f_k(\omega)e^{z\lambda_k(\omega)} \quad$}

\smallskip\noindent $(z\in\mathbb{C},\ \omega\in\Omega).$ Let $\sigma_{\text{conv}}(f,\omega)$ and $\sigma_a(f,\omega)$ be the abscissa of convergence and absolute convergence of this series for fixed $\omega\in\Omega,$  respectively. The simple modification of \cite{Mand}--\cite{ZadSka} one has that for the Dirichlet series $f\in\mathcal{D}$ for fixed $\omega\in\Omega$ such that $\lambda_k(\omega)\to +\infty\ (k\to +\infty)$
\begin{gather}\label{eq0}
\sigma_a(f,\omega)\leq\sigma_{\text{conv}}(f,\omega)\leq \alpha_0(\omega):=\varliminf\limits_{k\to
+\infty}\frac{-\ln|f_k(\omega)|}{\lambda_k(\omega)}
\leq\sigma_a(f,\omega)+\tau(\omega,\Lambda),
\end{gather}
or in the case $-\ln |f_k(\omega)|\to +\infty\ (k\to +\infty)$
\begin{gather}\label{eq01'}
(1-h)\sigma_{\text{conv}}(f,\omega)\leq (1-h)\alpha_0(\omega)
\leq\sigma_a(f,\omega),\quad h=h(\omega,\mathbf{f})
\end{gather}
where

\centerline{$\displaystyle
\tau(\omega,\Lambda):=\varlimsup_{k\to\infty}\frac{\ln k}{\lambda_k(\omega)},\quad h(\omega,\mathbf{f}):=\varlimsup_{k\to\infty}\frac{\ln k}{-\ln |f_k(\omega)|}.$}

\smallskip\noindent Also,
\begin{equation}\label{eq01}
\sigma_{\text{conv}}(f,\omega)=\sigma_a(f,\omega)=\alpha_0(\omega)
\end{equation}
for fixed $\omega\in\Omega$ such that $\tau(\omega)=0$ or
$\ln k/(-\ln |f_k(\omega)|)\to +0$\ $(k\to +\infty)$. Denote

\smallskip\centerline{$\sigma(f,\omega):=\sigma_a(f,\omega).$}

\smallskip\noindent
Remark, that from condition $\tau(\omega)<+\infty$ we get $\lambda_k(\omega)\to +\infty\ (k\to +\infty)$. From condition
$\sigma_{\text{conv}}(f,\omega)>0$ follows, that $-\ln |f_k(\omega)|\to +\infty\ (k\to +\infty)$, because in this case of the series of the form $f(0)=\sum_{k=0}^{+\infty}f_k(\omega)$ is convergent.

The following assertion is proved in \cite[Corollary 5]{ZadSka} (another version \cite[Theorem 1]{Mul}) in the case of the deterministic Dirichlet series with sequence of exponents that increase to infinity, i.e., $f_k(\omega)\equiv f_k\in\mathbb{C}$\ $(k\geq 0)$ and $\lambda_k(\omega)\equiv\lambda_k$, $0\leq\lambda_k <\lambda_{k+1}\to +\infty$\ $(0\leq k\to +\infty)$.

\begin{Proposition}\label{pro-01}\sl
Let $f\in\mathcal{D}$. Then $\sigma_{\text{a}}(f,\omega)\leq\sigma_{\text{conv}}(f,\omega)\leq\alpha_0(\omega)$\
$(\forall \omega\in\Omega)$, and
\begin{equation}\label{absc}
\sigma_a(f,\omega)\geq \gamma(\omega)\alpha_0(\omega)-\delta(\omega)\geq \gamma(\omega)\sigma_{\text{conv}}(f,\omega)-\delta(\omega)
\end{equation}
for arbitrary real random variables $\gamma, \delta$ and for all $\omega\in\Omega$ such that $\gamma(\omega)>0$ and
\begin{equation}\label{Zad}
\sum\limits_{k=0}^{+\infty}|f_k(\omega)|^{1-\gamma(\omega)}e^{-\delta(\omega)\lambda_k(\omega)}<+\infty.
\end{equation}
\end{Proposition}
\begin{proof}[Proof of Proposition \ref{pro-01}] It is obviously that $\sigma_{\text{a}}(f,\omega)\leq\sigma_{\text{conv}}(f,\omega)$.

We prove first that $\sigma_{\text{conv}}(f,\omega)\leq\alpha_0(\omega)$. Indeed, assume that $\alpha_0(\omega)\not=\infty$ and put $x_0=\alpha_0(\omega)+\varepsilon,$ where $\varepsilon>0$ is arbitrary. Then,
$$
|f_k(\omega)|e^{x_0\lambda_k(\omega)}=\exp\{\lambda_k(\omega)(\ln |f_k(\omega)|/\lambda_k(\omega)+x_0)\}.
$$
But by definition of $\alpha_0(\omega)$ there exists a sequence $k_j\to +\infty$\ $(j\to +\infty)$ such that
$$
\ln |f_k(\omega)|/\lambda_k(\omega)>-(\alpha_0(\omega)+\varepsilon/2)\quad (k=k_j,\ j\geq 1).
$$
Thus, $\ln |f_k(\omega)|/\lambda_k(\omega)+x_0>\varepsilon/2\quad (k=k_j,\ j\geq 1)$, and
$$
|f_k(\omega)|e^{x_0\lambda_k(\omega)}\geq e^{\lambda_k(\omega)\varepsilon/2}\geq 1\quad (k=k_j,\ j\geq 1),
$$
therefore $\sigma_{\text{conv}}(f,\omega)\leq\alpha_0(\omega)+\varepsilon$, but $\varepsilon>0$ is arbitrary.

The case $\alpha_0(\omega)=+\infty$ is trivial. In the case $\alpha_0(\omega)=-\infty$ for every $E>0$ and for some sequence $k_j\to +\infty$\ $(j\to +\infty)$ by definition $\alpha_0(\omega)$ we obtain
$$
\ln |f_k(\omega)|/\lambda_k> E\quad (k=k_j,\ j\geq 1).
$$
Therefore $|f_k(\omega)|\exp\{-E\lambda_k\}>1$\ $(k=k_j,\ j\geq 1)$, i.e. the Dirichlet series diverges at the point $z=-E$, but $E>0$ is arbitrary. Thus, $\sigma_{\text{conv}}=-\infty$.

Let now $x_0=\gamma(\omega)(\alpha_0(\omega)-\varepsilon)-\delta(\omega)$\ for arbitrary $\varepsilon>0$. Then,
\begin{gather}
|f_k(\omega)|e^{x_0\lambda_k(\omega)}=|f_k(\omega)|^{1-\gamma(\omega)}e^{-\delta(\omega)\lambda_k(\omega)}
|f_k(\omega)|^{\gamma(\omega)}e^{\gamma(\omega)(\alpha_0(\omega)-\varepsilon)\lambda_k(\omega)}=\nonumber\\
=|f_k(\omega)|^{1-\gamma(\omega)}e^{-\delta(\omega)\lambda_k(\omega)}
\bigg(|f_k(\omega)|e^{(\alpha_0(\omega)-\varepsilon)\lambda_k(\omega)}\bigg)^{\gamma(\omega)}.\label{tot1}
\end{gather}
By definition of $\alpha_0(\omega)$ we obtain
$$
\alpha_0(\omega)<\frac{-\ln f_k(\omega)}{\lambda_k(\omega)}+\varepsilon/2
$$
for $k\geq k_0(\omega)$,
and thus
$$
|f_k(\omega)|e^{(\alpha_0(\omega)-\varepsilon)\lambda_k(\omega)}<\exp\{-\lambda_k \varepsilon/2\}\leq 1\quad (k\geq k_0(\omega)).
$$
Hence by \eqref{tot1} one has
$$
|f_k(\omega)|e^{x_0\lambda_k(\omega)}\leq |f_k(\omega)|^{1-\gamma(\omega)}e^{-\delta(\omega)\lambda_k(\omega)}
$$
and by condition \eqref{Zad} we obtain
$$
\sigma_{a}\geq x_0=\gamma(\omega)(\alpha_0(\omega)-\varepsilon)-\delta(\omega).
$$
But, $\varepsilon>0$ is arbitrary.
\end{proof}
\begin{Remark}\label{rem1}\rm
Condition  \eqref{Zad} from Proposition \ref{pro-01} we get from following condition
\begin{equation}\label{Zad2}
h(\gamma, \delta, \omega):=
\varliminf\limits_{k\to\infty}\frac{(\gamma(\omega)-1)\ln\,|f_k(\omega)|+\delta(\omega)\lambda_k(\omega)}{\ln k}>1.
\end{equation}
Moreover, condition  \eqref{Zad} is weaker.
Also note, that condition \eqref{Zad} implies, that for such $\omega$
$$
(\gamma(\omega)-1)\ln\,|f_k(\omega)|+\delta(\omega)\lambda_k(\omega)\to +\infty\quad (k\to +\infty).
$$
But, in general, from this condition don't follows neither $\lambda_k(\omega)\to +\infty$ nor $\ln\,|f_k(\omega)|\to \infty\ (k\to +\infty)$.
\end{Remark}

From Proposition \ref{pro-01} simply follows such a statement.
\begin{Proposition}\label{pro-2}\sl Let $f\in\mathcal{D}$. Then equalities \eqref{eq01} hold for all $\omega\in\Omega$ such, that
\begin{equation}\label{coef}
\ln k=o(\ln |f_k(\omega)|)\quad (k\to +\infty).
\end{equation}
\end{Proposition}
\begin{proof}[Proof of Proposition \ref{pro-2}.]
For fixed $\omega$ we put $\varepsilon_k(\omega)\overset{def}=\frac{\ln k}{\ln|f_k(\omega)|}$, $N^+\!\!\overset{def}=\{k\colon\frac1{\varepsilon_k(\omega)}\geq 1\}$, $N^-\overset{def}=\{k\colon\frac1{\varepsilon_k(\omega)}\leq -1\}$. Then,
there exists $k_0=k_0(\omega)\in\mathbb{N}$ such that $k\in N^+\cap N^-$ for all $k> k_0$. So we can write
		\begin{gather*}
		\sum_{k=0}^{+\infty}f_k(\omega)e^{z\lambda_{k}(\omega)}=\sum_{k=0}^{k_0}f_k(\omega)e^{z\lambda_{k}(\omega)}+\sum_{k\in N^+}f_k(\omega)e^{z\lambda_{k}(\omega)}+\\ +\sum_{k\in N^-}f_k(\omega)e^{z\lambda_{k}(\omega)}=f_1+f_2+f_3.
		\end{gather*}
		Remark that for $f_2$  condition \eqref{Zad} holds with $\gamma(\omega)=\gamma_2(\omega)=1+\varepsilon,\ \delta(\omega)=0$ and for $f_3$ with $\gamma(\omega)=\gamma_3(\omega)=1-\varepsilon,\ \delta(\omega)=0$, where $0<\varepsilon<1$.
		Indeed, for $f_2$ we have
		\begin{gather*}
		h(\gamma, 0, \omega)=
		\varliminf\limits_{k\to\infty,\ k\in N^+}\frac{(\gamma(\omega)-1)\ln\,|f_k(\omega)|}{\ln k}=\\
		=\varliminf\limits_{k\to\infty,\ k\in N^+}\frac{\varepsilon\ln\,|f_k(\omega)|}{\ln k}=+\infty
		\end{gather*}
		and for  $f_3$
		\begin{gather*}
		h(\gamma, 0, \omega)=
		\varliminf\limits_{k\to\infty,\ k\in N^-}\frac{(\gamma(\omega)-1)\ln\,|f_k(\omega)|}{\ln k}=\\
		=\varliminf\limits_{k\to\infty,\ k\in N^-}\frac{-\varepsilon\ln\,|f_k(\omega)|}{\ln k}=+\infty.
		\end{gather*}
		Therefore condition \eqref{Zad} with $\delta(\omega)=0$ holds in both cases.
		Hence, by inequality \eqref{absc} from Proposition \ref{pro-01} we get
		\begin{gather*}
		\sigma(F,\omega)=\min\{\sigma(f_1,\omega), \sigma(f_2,\omega)\}\geq \min\{\gamma_2(\omega)\alpha^{+}(\omega), \gamma_3(\omega)\alpha^{-}(\omega)\}-\delta(\omega)\geq\\
		\geq\min\{\gamma_2(\omega)\alpha_0(\omega), \gamma_3(\omega)\alpha_0(\omega)\}
		\geq \min\{(1+\varepsilon)\alpha_0(\omega), (1-\varepsilon)\alpha_0(\omega)\},
		\end{gather*}
		where $$\alpha^{\pm}(\omega)=
		\varliminf\limits_{k\to\infty,\ k\in N^\pm}\frac{-\ln\,|f_k(\omega)|}{\lambda_k(\omega)}.$$
		Whence, using arbitrariness of $0<\varepsilon<1$ we get $\sigma(F,\omega)\geq \alpha_0(\omega)$. It remains to recall that $\sigma(F,\omega)\leq\sigma_{\text{зб}}(F,\omega)\leq\alpha_0(\omega)$.
\end{proof}

\begin{Remark}\rm If sequences $\Lambda$ and
$\mathbf{f}$ such that $\big(|f_k(\omega)|e^{x\lambda_{k}(\omega)}\big)$ is the sequences of independent random variables for every $x\in\mathbb{R}$, then by Kolmogorov's Zero-One Law (\cite{Kah2}) random variable $\sigma(f,\omega)$ is almost surely (a.s.) constant. That is $\sigma(f,\omega)=\sigma\in [-\infty,+\infty]$ a.s.
In the book \cite{Kah2} it is written when $\Lambda$ monotonic increasing to infinity sequence $\lambda_k(\omega)\equiv \lambda_k$.
The same we get when $\Big(\dfrac{-\ln |f_k(\omega)|}{\lambda_{k}(\omega)}\Big)$ is the sequences of independent random variables, and $\tau(\omega, \Lambda)=0$ or $h(\omega, \mathbf{f})=0$.  It follows by Proposition \ref{pro-2} from equalities \eqref{eq01}.
\end{Remark}

In the papers  \cite{Tian}--\cite{ska_shap_buk} considered question about abscissas of convergence
random Dirichlet series from the class  $\mathcal{D}$ in case, when $\Lambda_+=(\lambda_k)$ is increasing
sequence of positive numbers, i.e., $0=\lambda_0<\lambda_k<\lambda_{k+1}\to +\infty$\ $(1\leq k\to +\infty)$ and $\tau(\omega,\Lambda)\equiv\tau(\Lambda)<+\infty$.

%

We have such elementary assertion.

\begin{Proposition}\label{pro4}\sl Let  $f\in\mathcal{D}(\Lambda)$ be a Dirichlet series of the form
$$
f(z)=f(z,\omega)=\sum\limits_{k=0}^{+\infty}a_k Z_k(\omega)e^{z\lambda_k(\omega)},
$$
where $(Z_k(\omega))$ is a sequence of random complex-valued variables.

\noindent $1^0.$ If the condition $\tau(\omega,\Lambda)=0$  holds  and
\begin{equation}\label{eq3}
\lim\limits_{k\to +\infty}\dfrac{-\ln |Z_k(\omega)|}{\lambda_k(\omega)}=0\quad
\text{a.s.},
\end{equation}
then
$$
\sigma_{\text{conv}}(f,\omega)=\sigma_a(f,\omega)=\alpha_0(\omega)=\varliminf
\limits_{k\to+\infty}\dfrac{-\ln|a_k|}{\lambda_k(\omega)}\quad \text{ a.s.}
$$
$2^0.$\ If $\alpha_0(\omega)=+\infty$ and the conditions
$\tau(\omega,\Lambda)<+\infty$,
\begin{equation}\label{eq03}
\varliminf\limits_{k\to +\infty}\dfrac{-\ln |Z_k(\omega)|}{\lambda_k(\omega)}>-\infty\ \
\text{a.s.}
\end{equation}
hold, then
$\sigma_a(f,\omega)=+\infty$ a.s.
\end{Proposition}
Proposition \ref{pro4} we immediately obtain from inequalities \eqref{eq0}.

In the paper  \cite{TianF} is only  $1^0$ for the case of the Dirichlet series $f\in\mathcal{D}(\Lambda_+)$ of the form
$$
f(z)=f(z,\omega)=\sum\limits_{k=0}^{+\infty}a_k Z_k(\omega)e^{z\lambda_k}.
$$

 From Proposition \ref{pro4} in particular follows  theorems 1 (when $\alpha_0=+\infty$) and 3 (when $\alpha_0=0$) from \cite{TianF}, which are proved under such conditions for mathematical expectation
\begin{gather}\label{eq1}
(\exists \alpha>0)\colon\ \sup\{\mathbf{M}|Z_k|^{\alpha}\colon\ k\geq 0\}<+\infty,\\
\label{eq2}
(\exists \beta>0)\colon\ \sup\{\mathbf{M}|Z_k|^{-\beta}\colon\ k\geq 0\}<+\infty,
\end{gather}
 By the Bienayme-Chebyshev inequality (\cite{God,Sav}) and the Borel-Cantelli\break Lemma \big(\cite{Kah2}, also about refined Second Borel-Cantelli lemma see \cite{Erd_Ren}--\cite{bil}\big) from conditions \eqref{eq1} and \eqref{eq2} easy follows, that a.s. for all enough large $k$  inequalities $k^{-\gamma}\leq |Z_k(\omega)|<k^{\gamma}$\ with $\gamma=\max\{2/\alpha,2/\beta\}$ hold, and if $\tau(\Lambda)=0$, then and condition \eqref{eq3}. Similarly, if $\tau(\Lambda)<+\infty$, then from condition \eqref{eq1} follows condition \eqref{eq03}.

In papers \cite{Heden,Fil,Ding} in the case of independent random variables $\mathbf{f}=(f_k)$, besides, generalized on class $\mathcal{D}(\Lambda)$ assertion of known Blackwell's  conjecture on power series with random coefficients, proved in \cite{Ryll} (see also \cite{Kah2}).
Besides, in \cite{Fil} in case, when $f_k(\omega)=a_k Z_k(\omega)$,\ $\tau(\Lambda)<+\infty$,\ $0<c_1\leq |Z_k(\omega)|\leq c_2<+\infty$ a.s. and $\mathbf{M}Z_k=0$\ $(k\geq 0)$, proved, that $\sigma_{\text{conv}}(f,\omega)\leq \sigma_a(f,\omega)+\tau(\Lambda)/2$ a.s.

It should be noted, that condition  \eqref{eq3} can be replaced by condition on sequence of distribution functions of random variables $(|Z_k(\omega)|)$. Exactly,\break condition \eqref{eq3} follows from such condition (see \cite{ska_shap_buk})
$$
(\forall \varepsilon >0)\colon\ \sum\limits_{k=0}^{+\infty}P\{\omega\colon |\ln|Z_k(\omega)||> \varepsilon\lambda_k\}<+\infty,
$$
and the previous condition holds, if and only if for all $\varepsilon >0$ the series
$$
\sum\limits_{k=0}^{+\infty}\big(1-F^*_k(e^{\varepsilon\lambda_k})+F^*_k(e^{-\varepsilon\lambda_k})\big)<+\infty
$$
is convergent,
where $F_k^*(x):=P\{\omega\colon |Z_k(\omega)|<x\}$ is the distribution function of $|Z_k(\omega)|$. In particular, from this condition follows $\lim_{k\to +\infty}F_k^*(+0)=0$.

In general case, for Dirichlet series  from the class $\mathcal{D}(\Lambda_+)$ in \cite{ska_shap_buk} (see also similar results for random gap power series in \cite{Arn}--\cite{shap_ska_ijma}) are proved such two theorems.

\begin{Theorem}[\cite{ska_shap_buk}]\label{the1}\sl Let $f\in\mathcal{D}(\Lambda_+)$ and $\mathbf{f}=\big(f_k(\omega)\big)$ be a sequence such that $\big(|f_k(\omega)|\big)$ is the sequence of pairwise independent random variables with functions of distribution
$ F_k(x):=P\{\omega : |f_{k}(\omega)|<x\},\ x\in \mathbb R,\
k\geq 0. $ The following assertions are true:

\noindent {\bf a)}\ If $\sigma(\omega)=\sigma(f,\omega)\geq\rho\in (-\infty,
+\infty)$\ a.s.,\ then $(\forall \varepsilon>0)\colon\
\sum_{k=0}^{+\infty}(1-F_k((e^{-\rho}+\varepsilon)^{\lambda_k}))<\infty$.

\noindent {\bf b)}\ If exists a sequence $(\delta_k)$ such, that
$\delta_k> -\infty$\ $(k\geq 0)$, $\varliminf\limits_{k\to +\infty}\delta_k= e^{-\rho}$, $\rho\in (-\infty,+\infty],$\ and $
\sum_{k=0}^{+\infty}(1-F_k(\delta_k^{\lambda_k}))=+\infty $, then
$\sigma(f,\omega)\leq \rho$ a.s.
\end{Theorem}
\begin{Theorem}[\cite{ska_shap_buk}]\label{the2}\sl  Let $f\in\mathcal{D}(\Lambda_+)$ and $\mathbf{f}=\big(f_k(\omega)\big)$ be a sequence of random variables with functions of distribution
$ F_k(x),\ x\in \mathbb R,$\ of the random variables $|f_k(\omega)|$,\ $k\geq 0.$ The following assertions are true:

\noindent {\bf a)}\ If there exist $\rho\in (-\infty, +\infty)$ and a sequence $(\varepsilon_k)$
that $\varepsilon_k\to +0$\ $(k\to +\infty)$ and
$\sum_{k=0}^{+\infty}(1-F_k((e^{-\rho}+\varepsilon_k)^{\lambda_k}))<\infty$, then\
$\sigma(f,\omega)\geq \rho$\ a.s.

\noindent {\bf b)} If $\sigma(f,\omega)=-\infty$\ a.s.,\ then
 $(\forall E>1)\colon$\
$\sum_{k=0}^{+\infty}(1-F_k(E^{\lambda_k}))=+\infty$.
\end{Theorem}
In this paper we prove similar theorems for Dirichlet series with random exponents $(\lambda_k(\omega))$ and deterministic coefficients $\mathbf{f}=(f_k)$,\ $f_k\in\mathbb{C},$\ $k\geq 0$.

Note, in paper \cite{Nik} the methods of probability theory are used to prove a number of theorems on the behavior of Dirichlet series with independent exponents. The results so obtained are applied to the theory of the $ \zeta$-function and to the behavior of the solution of the wave equation as $ t\to\infty$.
In article \cite{Hol} a power series of the form
$
\sum\nolimits_{k=0}^{+\infty}z^{X_k(\omega)}
$
are studied, where $(X_k(\omega))$ is a strictly increasing integer-valued stochastic process.

\medskip\noindent {\bf  2.\ The main results:\  series with random exponents.}
In this section we assume, that $f_k(\omega)\equiv f_k\in\mathbb{C}$\ $(k\geq 0)$ and condition

\smallskip\centerline{$
\ln k=o(\ln |f_k|)\quad (k\to +\infty),
$}

\smallskip\noindent
holds, that condition \eqref{coef}  satisfies for all $\omega\in\Omega$, therefore by Proposition \ref{pro-2} equalities \eqref{eq01} for every $\omega\in\Omega$ hold.

\begin{Theorem}\label{the3}\sl Let $f\in\mathcal{D}(\Lambda)$ and $\Lambda=\big(\lambda_k(\omega)\big)$ be a sequence of pairwise independent random variables with distribution functions
$ F_k(x):=P\{\ \omega : \lambda_{k}(\omega)<x\},\ x\in \mathbb R,\
k\geq 0. $ The following assertions hold:

\noindent $\mathbf{i)}$\ $\sigma(\omega)=\sigma(f,\omega)\geq\rho\in (0,
+\infty)$\ a.s.\ $\Longrightarrow$ $(\forall \varepsilon\in (0,\rho) )\colon\
\sum_{k=0}^{+\infty}\big(1-F_k(\frac{\ln
|f_k|}{-\rho+\varepsilon})\big)$ $<\infty$;

\noindent $\mathbf{ii)}$\ $0\geq\sigma(\omega)=\sigma(f,\omega)\geq\rho\in (-\infty,0]$\ a.s.\ $\Longrightarrow$ $(\forall \varepsilon >0)\colon\
\sum_{k=0}^{+\infty}F_k(\frac{\ln
|f_k|}{-\rho+\varepsilon})<\infty$;
\end{Theorem}
\begin{proof}[Proof of Theorem \ref{the3}]
$\mathbf{i)}$ If $\sigma(f,\omega)\geq\rho\in (0,+\infty)$ a.s.,
then from \eqref{eq01} we have\

$$
(\exists B\in\mathcal{A}, P(B)=1)(\forall \omega \in B)\colon\
\varliminf\limits_{k\to+\infty}\frac{-\ln |f_{k}|}{\lambda_k(\omega)}\geq  \rho,
$$
and by definition of $\varliminf$,\
\begin{gather}\label{ineqA}
(\forall \omega \in B) (\forall \varepsilon\in(0,\rho) ) (\exists
k^*(\omega)\in\mathbb N) (\forall k\geqslant
k^*(\omega))\colon\ \lambda_{k}(\omega)<\frac 1{(-\rho+\varepsilon )}\ln
|f_k|.
\end{gather}
We denote $\displaystyle A_k:=\Big\{\omega\colon
\lambda_{k}(\omega)\geq\frac 1{(-\rho+\varepsilon )}\ln
|f_k|\Big\}. $ It is clear, that
$
B\subset
\overline{C}:=\bigcup_{N=0}^{\infty}\bigcap_{k=N}^{\infty} \overline{A}_k,
$
hence $P(\overline{C})=1$,
and $C=\bigcap_{N=0}^{\infty}\bigcup_{k=N}^{\infty} {A}_k$ is the event ``$({A}_k)$ infinitely
often'', i.e. $\overline{C}$ is the event  ``$({A}_k)$ finitely
often''.

From pairwise independence of random variables $(\lambda_k(\omega))$ follows pairwise independence of events $(A_k)$. Therefore, by refined Second Borel-Cantelli Lemma (\cite[p.84]{bil})
$$
\sum\limits_{k=0}^{+\infty}\Big(1-F_k\big({\ln
|f_k|}/{(-\rho+\varepsilon)}\big)\Big)=\sum\limits_{k=0}^{+\infty}P({A}_k)<+\infty.
$$

$\mathbf{ii)}$ If $0\geq\sigma(\omega,f)\geq\rho\in(-\infty,0]$ a.s., then
instead of \eqref{ineqA} we obtain
\begin{gather*}
(\forall \omega \in B, P(B)=1) (\forall \varepsilon >0) (\exists
k^*(\omega)\in\mathbb N) (\forall k\geq
k^*(\omega))\colon\\
\lambda_{k}(\omega)>\frac 1{(-\rho+\varepsilon )}\ln
|f_k|.
\end{gather*}
Therefore, for $\displaystyle A_k:=\Big\{\omega\colon
\lambda_{k}(\omega)\leq\frac 1{(-\rho+\varepsilon )}\ln
|f_k|\Big\} $  by the Second  Borel-Cantelli lemma we obtain again
$$
\sum\nolimits_{k=0}^{+\infty}F_k\big({\ln
|f_k|}/{(-\rho+\varepsilon)}\big)=\sum\nolimits_{k=0}^{+\infty}P(A_k)<+\infty.
$$
This completes the proof of Theorem \ref{the3}.
\end{proof}
\begin{Remark}\label{re3}\rm If $\sigma(f,\omega)>\rho\in [0,+\infty)$ a.s.,
then from \eqref{eq01} by definition of $\varliminf$ we have\
\begin{gather*}
(\forall \omega \in B) (\exists \varepsilon^*=\varepsilon^*(\omega)> 0 ) (\exists
k^*(\omega)\in\mathbb N) (\forall k\geqslant
k^*(\omega))\colon\\
 \lambda_{k}(\omega)<\frac 1{-(\rho+\varepsilon^* )}\ln
|f_k|,
\end{gather*}
and similarly as in proof of\ $\mathbf{i)}$ we obtain
$\sum_{k=0}^{+\infty}\big(1-F_k(\frac{-\ln
|f_k|}{\rho})\big)<+\infty$
in the case $\rho>0$ and in the case $\rho=0$ one has $
\sum_{k=0}^{+\infty}\big(1-F_k(+0)\big)<+\infty,
$
i.e., in particular, $\lim_{k\to +\infty}F_k(+0)=1$. Namely, if $\varliminf\limits_{k\to +\infty}F_k(+0)<1$, then $\sigma(f,\omega)\leq 0$ a.s. 
\end{Remark}
\begin{Theorem}\label{the3'}\sl Let  $\Lambda=\big(\lambda_k(\omega)\big)$ be a sequence of random variables with distribution functions
$ F_k(x):=P\{\ \omega : \lambda_{k}(\omega)<x\},\ x\in \mathbb R,\
k\geq 0, $ and $f\in\mathcal{D}(\Lambda)$.\ The following assertions hold:

\noindent $\mathbf{i)}$\ If there exist $\rho\in (0, +\infty)$ and a sequence $(\varepsilon_k)$
such that $\varepsilon_k\to +0$\ $(k\to +\infty)$ and
$\sum_{k=0}^{+\infty}\big(1-F_k(\frac{\ln |f_k|}{-\rho+\varepsilon_k})\big)<+\infty$, then\
$\sigma(f,\omega)\geq \rho$\
a.s.

\noindent $\mathbf{ii)}$\ If there exist $\rho\in (-\infty, 0]$ and a sequence $(\varepsilon_k)$
such that $\varepsilon_k\to +0$\ $(k\to +\infty)$ and
$\sum_{k=0}^{+\infty}F_k(\frac{\ln |f_k|}{-\rho+\varepsilon_k})<+\infty$, then\
$\sigma(f,\omega)\geq \rho$\
a.s.

\end{Theorem}
\begin{proof}[Proof of Theorem \ref{the3'}]
$\mathbf{i)}$\
We denote $A_k=\big\{\omega\colon
\lambda_k(\omega)\geq\frac{\ln |f_k|}{-\rho+\varepsilon_k}\big\}$. Taking into account that
$1-F_k(\frac{\ln |f_k|}{-\rho+\varepsilon_k})=P(A_k)$, from condition
we obtain, that\break $\sum\nolimits_{k=0}^{+\infty}P(A_k)<\infty$, and by the first part of Borel-Cantelli Lemma $P(\overline C)=1$,\ $C:=\bigcap_{N=0}^{\infty}\bigcup_{k=N}^{\infty}{A}_k$.
Since,
$\overline{C}=\bigcup_{N=0}^{\infty}\bigcap_{k=N}^{\infty}\overline{A}_k$, then for all
$\omega\in \overline{C}$ there exists $k=k^*(\omega)$ such that $\omega\in
\overline{A}_k$ and $-\rho+\varepsilon_k<0$ for all $k\geq k^*(\omega)$. Here, $(\forall k\geq
k^*(\omega))\colon \lambda_k(\omega)<\frac{\ln |f_k|}{-\rho+\varepsilon_k}.$ Using $\frac{-\ln |f_k|}{\lambda_k(\omega)}>\rho-\varepsilon_k$, we get
\begin{gather}\label{fina}
\sigma(f,\omega)=\varliminf\limits_{k\to
+\infty}\frac{-\ln |f_{k}|}{\lambda_k(\omega)}
\geq \varliminf_{k\to
+\infty}(\rho -\varepsilon_k)=\rho\quad \text{a.s.}
\end{gather}

\noindent $\mathbf{ii)}$ We denote $\displaystyle A_k:=\Big\{\omega\colon
\lambda_{k}(\omega)<\frac{\ln |f_k|}{-\rho+\varepsilon_k}\Big\}.$ By the condition $\sum_{k=0}^{+\infty}P(A_k)<+\infty$. Since, by the first part of Borel-Cantelli lemma $P(\overline C)=1$,\break $C:=\bigcap_{N=0}^{\infty}\bigcup_{k=N}^{\infty}{A}_k$.
Where, as above for every
$\omega\in\overline{C}=\bigcup_{N=0}^{\infty}\bigcap_{k=N}^{\infty}\overline{A}_k$ there exists $k=k^*(\omega)$ such that $\omega\in
\overline{A}_k$ and $-\rho+\varepsilon_k>0$ for all $k\geq k^*(\omega)$, such hat, $(\forall k\geq
k^*(\omega))\colon \lambda_k(\omega)\geq\frac{\ln |f_k|}{-\rho+\varepsilon_k}.$ Hence, $\frac{-\ln |f_k|}{\lambda_k(\omega)}>\rho-\varepsilon_k$ and, therefore, we have again the ``chain'' of relations \eqref{fina}.

The proof of Theorem \ref{the3'} is complete.
\end{proof}

\medskip\noindent{\bf 2.\ Some corollaries.}

\begin{Corollary}\label{cor1}\sl Let $f\in\mathcal{D}(\Lambda)$ and $\Lambda=(\lambda_k(\omega))$ be a sequence of pairwise independent random variables with distribution functions
$ F_k(x),\
k\geq 0. $ If $\varliminf\limits_{k\to +\infty}F_k(+0)<1$ and $ f_k\to 0$\ $(k\to +\infty)$, then
$\sigma(f,\omega)= 0$ a.s.
\end{Corollary}
\begin{proof}[Proof of Corollary \ref{cor1}] By Remark \ref{re3},\ $\sigma(f,\omega)\leq 0$ a.s. It is remains to prove that $\sigma(f,\omega)\geq 0$ a.s. Indeed, $\lambda_k(\omega)\geq 0$, therefore $F_k(0)=P\{\omega\colon \lambda_k(\omega)<0\}=0$. Hence, $\sum_{k=k_0}^{+\infty}F_k(\frac{\ln |f_k|}{\varepsilon_k})<+\infty$ because $\frac{\ln |f_k|}{\varepsilon_k}<0$\ $(k\geq k_0)$ and by Theorem \ref{the3'} $\mathbf{ii)}$, $\sigma(f,\omega)\geq 0$ a.s.
\end{proof}

\begin{Corollary}\label{cor2}\sl  Let $f\in\mathcal{D}(\Lambda)$ and $\Lambda=(\lambda_k(\omega))$ be a sequence of pairwise independent random variables with distribution functions
$ F_k(x),\
k\geq 0. $ If there exists a positive random variable
$a(\omega)$ such that $(\forall x\geq 0)(\forall
k\in\mathbb{Z}_{+})\colon$ $F_k(x)\leq F_{a}(x):=P\{\omega\colon
a(\omega)<x \}$ and $F_{a}(+0)< 1$ and $ f_k\to 0$\ $(k\to +\infty)$, then $\sigma(f,\omega)= 0$ a.s.
\end{Corollary}
{\it The statement of Corollary \ref{cor2}} follows immediately from Corollary \ref{cor1}.

\begin{Corollary}\label{cor3}\sl Let $f\in\mathcal{D}(\Lambda)$ and $\Lambda=(\lambda_k(\omega))$ be a sequence of random variables with distribution functions
$ F_k(x),\
k\geq 0. $
If $ f_k\to 0$\ $(k\to +\infty)$ and there exist a positive random variable $b(\omega)$  and  $\rho>0$ such that
$(\forall x\geq 0)(\forall k\in\mathbb{Z}_{+})\colon$ $F_k(x)\geq
F_{b}(x):=P\{\omega\colon b(\omega)<x \}$,
$\int_{0}^{+\infty}n_{\mu}(t\rho)\ dF_{b}(t)< +\infty$, where $n_{\mu}(t)=\sum\nolimits_{\mu_k\leq t}1$ is the counting function of a sequence $\mu_k=-\ln|f_k|$, then
$\sigma(f,\omega)\geq \rho$ a.s.
\end{Corollary}
\begin{proof}[Proof of Corollary \ref{cor3}] We remark that
\begin{gather*}
\sum_{k=k_0}^{n}\Big(1-F_k(\frac{\ln |f_k|}{-\rho+\varepsilon_k})\Big)\leq\sum_{k=k_0}^{n}\Big(1-F_k(\frac{-\ln |f_k|}{\rho})\Big)\leq\\
\leq\int\limits_{\mu_{k_0}}^{\mu_{n}}\Big(1-F_k\big(t/\rho\big)\Big)dn_{\mu}(t)
\leq\int\limits_{\mu_{k_0}}^{\mu_{n}}(1-F_b(t/\rho))dn_{\mu}(t)+O(1)=\\
=\int\limits_{\mu_{k_0}}^{\mu_n}n_{\mu}(t)\ dF_{b}(t/\rho)+O(1)=\int\limits_{{\mu_{k_0}}/{\rho}}^{{\mu_n}/{\rho}}n_{\mu}(t\rho)\ dF_{b}(t)+O(1)
\end{gather*}
as $n\to +\infty$, because $-\ln |f_k|>0$ for $k\geq k_0$ and $\rho-\varepsilon_k<\rho$ for all $k\geq 0$.
Thus $\sum_{k=k_0}^{+\infty}\Big(1-F_k(\frac{\ln |f_k|}{-\rho+\varepsilon_k})\Big)<+\infty$.
Hence by  Theorem \ref{the3'} $\mathbf{ii)}$ we complete the proof.
\end{proof}
\begin{Corollary}\label{cor4}\sl
Let  $\Lambda=(\lambda_k(\omega))$ be a increasing (a.s.) sequence of pairwise independent random variables and $f\in\mathcal{D}(\Lambda)$. If $F_0(+0)<1$, where $F_0$ is distribution function of $\lambda_0(\omega)$, and $f_k\to 0$ $(k\to +\infty)$, then $\sigma(f,\omega)= 0$ a.s.
\end{Corollary}
\begin{proof}[Proof of Corollary \ref{cor4}] We remark that $F_{k+1}(x)=P\{\omega\colon \lambda_{k+1}(\omega)<x\}\leq$\break $\leq P\{\omega\colon \lambda_{k}(\omega)<x\}= F_{k}(x)\leq    \ldots\leq F_{0}(x)$, because $\lambda_k(\omega)\leq \lambda_{k+1}(\omega)$ $(k\geq 0)$ a.s. Therefore, by Corollary \ref{cor2} we obtain the conclusion of Corollary \ref{cor4}.
\end{proof}

\end{document}